\documentclass[reqno,10pt]{amsart}

\usepackage{amsmath,amssymb,amsthm,mathrsfs}
\usepackage{longtable}
\usepackage{enumerate}
\usepackage{graphicx} 
\usepackage{wasysym}

\usepackage{xcolor}

\allowdisplaybreaks

\theoremstyle{plain}
\newtheorem{theorem}{Theorem}[section]
\newtheorem{lemma}[theorem]{Lemma}
\newtheorem{proposition}[theorem]{Proposition}

\numberwithin{equation}{section}
\numberwithin{figure}{section}

\begin{document}

\title[Laguerre-Sobolev orthogonal Polynomials]{Laguerre-Sobolev orthogonal Polynomials and Boundary Value Problems on a semi-infinite domain}

\author[C. F. Bracciali and M. A.  Pi\~nar]
{Cleonice F. Bracciali and Miguel A. Pi\~nar}

\address[C. F. Bracciali]{Departamento de Matem\'{a}tica, IBILCE, UNESP - Universidade Estadual Paulista,
15054-000, S\~ao Jos\'e do Rio Preto, SP (Brazil)}
\email{cleonice.bracciali@unesp.br}

\address[M. A.  Pi\~nar]{Instituto de Matem\'aticas UGR -- IMAG \&
Departamento de Matem\'{a}tica Aplicada, Facultad de Ciencias. Universidad de Granada (Spain)}
\email{mpinar@ugr.es}

\thanks{This work was started while Miguel A. Pi\~nar was visiting Cleonice F. Bracciali at the Department of Mathematics of UNESP, Universidade Estadual Paulista, at S\~ao Jos\'e do Rio Preto, SP, Brazil. His stay in UNESP for a period of one month during october 2023, was supported by the Brazilian Federal Agency for Support and Evaluation of Graduate Education (CAPES), in the scope of the CAPES-PrInt International Cooperation Program, process number 88887.890314/2023-00. This author (MAP) is extremely grateful to the Department of Mathematics of UNESP for receiving all the necessary support to undertake this research.
The author CFB thanks the support of grants 88881.702910/2022-01 from CAPES and 2022/09575-5 from FAPESP of Brazil. The author MAP thanks grant PID2023.149117NB.I00 funded by MICIU/AEI/ 10.13039/501100011033 and ERDF \textit{A way of making Europe}.}

\date{\today}

\begin{abstract}
We study a family of Laguerre--Sobolev orthogonal polynomials associated with a Sobolev inner product arising from second--order boundary value problems on the semi--infinite interval $(0,+\infty)$. These polynomials generate an orthogonal basis of test functions vanishing at the endpoints and are especially well suited for the spectral approximation of Schr\"odinger--type problems with singular potentials. Explicit connection formulas with classical Laguerre polynomials are obtained, together with recurrence relations and asymptotic properties of the corresponding coefficients. A generating function involving Bessel functions is also derived. As an application, we develop a fully diagonalized Laguerre--Sobolev spectral method for Dirichlet problems with singular potentials. The method avoids the solution of linear systems and can be implemented recursively. Numerical experiments for a Schr\"odinger--type equation with inverse--distance potential confirm spectral accuracy and exponential convergence.
\end{abstract}


.

\maketitle

%%%%%%%%%%%%%%%%%%%%%%%%%%%%%%%%%%%%%%%%%%%%%%%%%%%%%%%%%%%%%%%%%%%%%%%%%%%%%%
\section{Introduction}
%%%%%%%%%%%%%%%%%%%%%%%%%%%%%%%%%%%%%%%%%%%%%%%%%%%%%%%%%%%%%%%%%%%%%%%%%%%%%%

Spectral methods based on orthogonal polynomials are a well--established tool for the numerical solution of differential equations, offering high accuracy and exponential convergence for smooth solutions; see, e.g., \cite{BM97,Canuto06,Funaro}. For problems posed on semi--infinite intervals, Laguerre and generalized Laguerre polynomials provide a natural approximation framework and have been extensively used in spectral and pseudospectral methods on $(0,+\infty)$ \cite{Boyd01,SW09}.  

Despite their effectiveness, classical Laguerre spectral methods may become less efficient when boundary conditions at infinity or singular coefficients are present in the differential operator. In such cases, Sobolev orthogonal polynomials offer a valuable alternative, as they incorporate derivative information and boundary behavior directly into the approximation space. This approach has led to the development of diagonalized spectral methods that avoid the solution of linear systems and significantly reduce computational complexity; see, for example, \cite{Ai18,FMPP24,LWL23}.  

In this paper, we introduce and analyze a family of Laguerre--Sobolev orthogonal polynomials associated with a Sobolev inner product naturally induced by a second--order differential operator on the half--line. The corresponding orthogonal functions form a basis of the space $\mathbb{P}\,x\,e^{-x/2}$ and satisfy homogeneous Dirichlet boundary conditions at $0$ and $+\infty$. This makes them particularly well suited for the spectral approximation of Schr\"odinger--type boundary value problems with singular potentials.  

We establish explicit connection formulas between the proposed Sobolev orthogonal polynomials and the classical Laguerre polynomials, derive recurrence relations for the connection coefficients, and obtain their asymptotic behavior using refined ratio asymptotics for Laguerre polynomials. These analytical results allow us to derive relative asymptotics and a closed generating function expressed in terms of Bessel functions.  

As an application, we develop a fully diagonalized Laguerre--Sobolev spectral method for second--order Dirichlet boundary value problems on $(0,+\infty)$. The resulting algorithm computes the Fourier--Sobolev coefficients recursively and requires only the evaluation of integrals involving classical Laguerre polynomials, leading to computational costs comparable to standard Laguerre methods while providing improved stability for singular problems. Numerical experiments for a non--homogeneous Schr\"odinger equation with inverse--distance potential confirm the spectral accuracy and exponential convergence of the proposed method.  

The paper is organized as follows. In Section~2 we recall basic properties of Laguerre polynomials. Section~3 is devoted to the construction and analysis of the Laguerre--Sobolev orthogonal polynomials, including connection formulas, asymptotic results and a generating function. In Section~4 we present the diagonalized spectral method for boundary value problems on the half--line. Numerical experiments illustrating the performance of the method are reported in Section~5.

%%%%%%%%%%%%%%%%%%%%%%%%%%%%%%%%%%%%%%%%%%%%%%%%%%%%%%%%%%%%%%%%%%%%%%
\section{
Laguerre orthogonal polynomials
}
%%%%%%%%%%%%%%%%%%%%%%%%%%%%%%%%%%%%%%%%%%%%%%%%%%%%%%%%%%%%%%%%%%%%%%%%

In this paper, we deal with the sequence of polynomials $\{L^{(\alpha)}_n\}_{n\geqslant0}$, $\alpha > -1,$ orthogonal with respect to the classical Laguerre inner product defined by
\begin{equation}\label{Laguerre_ip}
\langle f, g\rangle_{\alpha}
= \int_{0}^{+\infty} f(t) g(t)  t^\alpha\,e^{-t} dt,
\end{equation}
satisfying  $L^{(\alpha)}_n(0) = \binom{n+\alpha}{n}$. The properties of Laguerre polynomials are obtained from \cite{Szego75}.

Square of the norms $\|L^{(\alpha)}_n\|^2 =
\langle L^{(\alpha)}_n, L^{(\alpha)}_n\rangle_{\alpha}$ are given by
\begin{equation}\label{norm-Laguerre}
\|L^{(\alpha)}_n\|^2 = \frac{\Gamma(n+\alpha+1)}{n!}.
\end{equation}

The hypergeometric representation of orthogonal Laguerre polynomials can be expressed as

\begin{equation}\label{exp_exp_Poc}
L_n^{(\alpha)}(x) = \sum_{k=0}^n (-1)^k  \binom{n+\alpha}{n-k} \frac{x^{k}}{k!}.
\end{equation}

Using the hypergeometric representation \eqref{exp_exp_Poc}, we can see that the leading coefficient of Laguerre polynomials does not depend on $\alpha \in \mathbb{R}$, in fact
$$
L_n^{(\alpha)}(x) = (-1)^n \frac{x^n}{n!} + \ldots.
$$

Laguerre orthogonal polynomials $\{L_n^{(\alpha)}\}_{n\geqslant 0}$ satisfy the three-term recurrence relation
\begin{equation}\label{ttrr}
\begin{aligned}
& (n+1)\,L_{n+1}^{(\alpha)}(x) = (2n+1 + \alpha -x)L_{n}^{(\alpha)}(x) - (n+\alpha) \,L_{n-1}^{(\alpha)}(x), \quad n\geqslant 0,\\
& L_{-1}^{(\alpha)}(x) = 0, \quad L_0^{(\alpha)}(x) = 1.
\end{aligned}
\end{equation}

Laguerre polynomials satisfy the derivation formula
\begin{equation}\label{der_formula}
\frac{d}{dx}\,L^{(\alpha)}_n(x) = - \,L^{(\alpha+1)}_{n-1}(x), \quad n\geqslant 1.
\end{equation}

Moreover, Laguerre polynomials satisfy the structure relation
\begin{align}
&L^{(\alpha)}_{n} (x) = L^{(\alpha+1)}_{n} (x) - L^{(\alpha+1)}_{n-1} (x), \label{str_rel2}
\end{align}
for $n\geqslant 1$.

Outer strong asymptotics of Laguerre polynomials in the complex plane 
(in $\mathbb{C}\setminus\mathbb{R}_{+}$) as $n\to\infty$ is well known in the literature 
and usually is given by Perron's formula. For $\alpha>-1$ we get
\begin{align}
L_{n}^{(\alpha )}\left( z\right) =&\frac{1}{2\sqrt{\pi}}e^{z/2}\left(
-z\right) ^{-\alpha /2-1/4}n^{\alpha /2-1/4}e^{2\sqrt{ -nz}}\nonumber\\
& \times \left\{ \sum\limits_{k=0}^{d-1}C_{k}(z)n^{-k/2}+\mathcal{O}%
(n^{-d/2})\right\} \label{Perron}.
\end{align}

Here $C_{k}(z)$ is independent of $n$, and $C_0(k)=1$. This relation holds for $z$\ in the
complex plane with a cut along the positive real semiaxis. The bound for the
remainder holds uniformly in every closed domain of the complex plane with
empty intersection with $x>0$ (see \cite{Szego75}, Theorem 8.22.3).

As a direct consequence of \eqref{Perron} we have that for an arbitrary real $\alpha$
\begin{equation} \label{asymp_Lagguerre_1}
\lim_{n\to\infty} \frac{L_{n}^{(\alpha)} (z)}{L_{n-1}^{(\alpha)} (z)} = 1,
\end{equation}
uniformly on compact subsets of $\mathbb{C}\setminus [0,+\infty)$.

However, when dealing with ratio asymptotics usually relation \eqref{asymp_Lagguerre_1}
is not accurate enough. In \cite[Theorem 3]{DHM13} a more accurate asymptotic expansion as $n\rightarrow \infty $ of arbitrary ratios of Laguerre polynomials has been obtained.

\begin{theorem}
Let $\alpha,\beta>-1$, $j\in\mathbb{Z}$ and $z\in\mathbb{C}\setminus [0,\infty)$, and fix an integer $d\geq 1$, then the ratio of arbitrary Laguerre polynomials has the following asymptotic expansion as $n\to\infty$:
\begin{equation} \label{asymp_expan}
\frac{L_{n+j}^{(\alpha)}(z)}{L_{n}^{(\beta)}(z)}
=\left(-\frac{z}{n}\right)^{\frac{\beta-\alpha}{2}}\sum_{m=0}^{d-1} U_m(\alpha,\beta,j,z)n^{-m/2}+\mathcal{O}(n^{-d/2}),
\end{equation}
where the first coefficients are
\begin{align*}
U_0(\alpha,\beta,j,z)&=1,\\
U_1(\alpha,\beta,j,z)&=\frac{\beta^2-\alpha^2+2z(\beta-\alpha-2j)}{4\sqrt{-z}}.
\end{align*}

The error term holds uniformly for $z$ in compact sets of $\mathbb{C}\setminus[0,\infty)$.

\end{theorem}

\medskip
 
For generating functions involving Laguerre polynomials, the following  result is known as Hardy-Hille formula 
(see \cite[Theorem 5.1, p. 102]{Szego75}).

\begin{theorem} \label{lema_gen} 
Let $J_{\alpha}$ be the Bessel function of the first kind of order $\alpha$, then
\begin{align}
\sum_{n=0}^{\infty} &\binom{n+\alpha}{n}^{-1}
L^{(\alpha)}_{n}(x) L^{(\alpha)}_{n}(y) \omega^n = \nonumber \\
&= \frac{\Gamma(\alpha+1)}{1-\omega} e^{- \frac{(x+y)\omega}{1-\omega}
} (-xy\omega)^{-\frac{\alpha}{2}} J_{\alpha} \left(  \frac{2(-xy\omega)^{\frac{1}{2}}}{1-\omega}\right). \label{gen_fun_Laguerre}
\end{align}
\end{theorem}.

%%%%%%%%%%%%%%%%%%%%%%%%%%%%%%%%%%%%%%%%%%%%
\section{
Laguerre--Sobolev orthogonal polynomials
}
%%%%%%%%%%%%%%%%%%%%%%%%%%%%%%%%%%%%%%%%%%%%

The problem we wish to address is the approximate solution of the boundary value problem (BVP, in short) for an ordinary differential equation associated with a stationary Schr\"odinger equation with a potential $V (x) = 1/x$,
\begin{equation}\label{BVP}
\begin{aligned}
& - u '' + \lambda\,\frac{1}{x}\,u = f(x), \\
& u(0)= \lim_{x \to +\infty}u(x) = 0,
\end{aligned}
\end{equation}
where $\lambda >0$. This problem, as usual, can be studied from a variational perspective by considering the Sobolev inner product
\begin{equation}\label{Sobolev_IP}
\langle u, v\rangle_\lambda = \lambda \,\int_{0}^{+\infty} u(x)\,v(x)\,\frac{1}{x}\,dx + \int_{0}^{+\infty} u'(x)\,v'(x)\,dx,
\end{equation}
associated with the variational formulation of \eqref{BVP}.

Let $\mathbb{P}$ denote the linear space of the real polynomials. The test functions for \eqref{BVP} should be chosen in the linear space $\mathbb{P}\, x\,e^{-x/2}$ of functions vanishing at the ends of the interval $[0,+\infty)$.  A basis for the space $\mathbb{P}\, x\,e^{-x/2}$ can be constructed from a family of orthogonal polynomials connected to the Sobolev inner product \eqref{Sobolev_IP}. In this section, we generate such orthogonal basis and study its properties.

Let us denote by $\{ S_{n}\, x\,e^{-x/2}\}_{n\geqslant0}$ the sequence of orthogonal functions obtained by applying the Gram-Schmidt process to the sequence
$\{x^k\, x \, e^{-x/2}\}_{k\geqslant0}$,
with respect to the Sobolev inner product $\langle \cdot, \cdot \rangle_\lambda$ defined in \eqref{Sobolev_IP}, where the polynomial $S_n(x)$ has the same leading coefficient as the Laguerre polynomial
$L_{n}^{(1)} (x)$, in this way
$$
S_{n}(x) = (-1)^n \frac{x^n}{n!} + \ldots
$$

Integration by parts gives
\begin{align*}
\langle S_n & x\,e^{-x/2},   S_m x\,e^{-x/2}\rangle_\lambda = \lambda \int_{0}^{+\infty} S_n(x) S_m(x) x\,e^{-x} dx\\
&+ \int_{0}^{+\infty} [S_n(x)(1-\frac{x}{2} + x S'_n(x)][S_m(x)(1-\frac{x}{2} + x S'_m (x)]\,e^{-x}dx\\
= & \int_{0}^{+\infty} ( 1 + \lambda - \frac{1}{4} x) S_n(x) S_m(x)\, x\,e^{-x}dx \\
& +\int_{0}^{+\infty} S'_n(x) S'_m(x) x^2\,e^{-x} dx.
\end{align*}

Therefore, the sequence of polynomials $\{S_{n}\}_{n\geqslant0}$ is orthogonal with respect to the Sobolev inner product
\begin{equation}\label{s1}
\begin{aligned}
\langle p, q \rangle_S= &\int_{0}^{+\infty} p(x) q(x)  ( 1 + \lambda - \frac{1}{4} x) \, x\,e^{-x}dx
\\
& + \int_{-1}^{1} p'(x) q'(x) x^2\,e^{-x} dx.
\end{aligned}
\end{equation}

Observe that, despite the polynomial $1 + \lambda - \frac{1}{4} x$ changes his sign inside the interval $[0, +\infty)$, the bilinear form \eqref{s1} is still positive because
$$
\langle p(x), p(x) \rangle_S = \langle p(x) x\,e^{-x/2},   p(x) x\,e^{-x/2}\rangle_\lambda >0,
$$
for $p(x) \neq 0$.

\medskip

First, we are going to obtain a connection formula between test functions and classical Laguerre polynomials.

\begin{proposition}\label{connect}
The following relation holds:
\begin{equation}\label{conechat}
L_{n}^{(1)} (x)= S_{n}(x) + {a}_{n-1}S_{n-1}(x), \quad n\geqslant 1,
\end{equation}
with
\begin{equation} \label{a_n}
a_{n-1} = \frac{n}{4}\frac{\|L^{(1)}_n\|^2}{\langle S_{n-1}, S_{n-1} \rangle_S} > 0.
\end{equation}

\end{proposition}

\begin{proof}
The Fourier expansion of $L_n^{(1)}$ in terms of the Sobolev orthogonal polynomials $\{S_n\}_{n\geqslant 0}$ yields
$$
L_n^{(1)}(x)= S_n(x)+\sum_{i=0}^{n-1} a_{n,i} S_i(x),
$$
where
$$
a_{n,i}=\frac{\langle L_n^{(1)}, S_i \rangle_S }{\|S_i \|_S^2}.
$$
Taking into account \eqref{s1} and the orthogonality relations of the Laguerre polynomials $L^{(1)}_{n}$ we get, for $i\leqslant n-1$,
$$
\begin{aligned}
\langle L_n^{(1)}, S_i \rangle_S = &  \int_{0}^{+\infty} L^{(1)}_{n}(x) S_i(x) ( 1 + \lambda - \frac{1}{4} x) \, x\,e^{-x}dx  \\
& - \int_{0}^{+\infty} L^{(2)}_{n-1}(x) S'_i(x) \, x^2\,e^{-x}dx \\
= & \int_{0}^{+\infty} L^{(1)}_{n}(x) S_i(x) ( 1 + \lambda - \frac{1}{4} x) \, x\,e^{-x}dx.
\end{aligned}
$$
Therefore $a_{n,i}=0, i<n-1$ and \eqref{conechat} holds.

For $i = n-1$
$$
\begin{aligned}
\langle L_n^{(1)}, S_{n-1} \rangle_S = &  \int_{0}^{+\infty} L^{(1)}_{n}(x) S_{n-1}(x) ( 1 + \lambda - \frac{1}{4} x) \, x\,e^{-x}dx  \\
& - \int_{0}^{+\infty} L^{(2)}_{n-1}(x) S'_{n-1}(x) \, x^2\,e^{-x}dx \\
= & - \frac{1}{4}\int_{0}^{+\infty} L^{(1)}_{n}(x) S_{n-1}(x) \, x \, x\,e^{-x}dx\\
= & \frac{n}{4} \,  \|L^{(1)}_n\|^2 >0.
\end{aligned}
$$
\end{proof}

Sobolev polynomials which are orthogonal with respect to \eqref{s1} are related to the nondiagonal Laguerre--Sobolev orthogonal polynomials defined in \cite{MMB97}. Both families satisfy a similar connection formula \eqref{conechat}, but the similarities end there since their asymptotic properties differ.

\medskip

\begin{lemma} The sequence $\{a_n\}_{n\geq 0}$ in \eqref{conechat} satisfies the recurrence relation
\begin{equation} \label{rec_rel_an}
a_n = \frac{n+2}{4 \lambda + 2(n+1) - n a_{n-1}}
\end{equation}
with
$$
a_0 = \frac{2}{4 \lambda + 2},
$$
and $0 < a_n < 1$ for all $n = 0, 1, \ldots$
\end{lemma}

\begin{proof}
Iterating relation \eqref{conechat} we get
\begin{align*}
L_{n+1}^{(1)} (x) &= S_{n+1}(x) + {a}_{n}S_{n}(x)\\
&= S_{n+1}(x) + {a}_{n}L_{n}^{(1)}(x) - {a}_{n-1}{a}_{n}S_{n-1}(x),
\end{align*}
which gives
\begin{align*}
\langle L_{n+1}^{(1)} (x), L_{n}^{(1)}(x) \rangle_S &= \langle S_{n+1}(x) , L_{n}^{(1)}(x) \rangle_S\\
&+ a_n \langle L_{n}^{(1)}(x), L_{n}^{(1)}(x) \rangle_S - {a}_{n-1}{a}_{n} \langle S_{n-1}(x), L_{n}^{(1)}(x) \rangle_S.
\end{align*}
Now, a direct computation shows
\begin{align*}
\langle S_{n+1}(x), L_{n}^{(1)}(x) \rangle_S &= 0, \\
\langle L_{n+1}^{(1)} (x), L_{n}^{(1)}(x) \rangle_S &= \frac{n+1}{4} \,  \|L^{(1)}_{n+1}\|^2,\\
\langle S_{n-1}(x), L_{n}^{(1)}(x) \rangle_S &= \frac{n}{4} \,  \|L^{(1)}_n\|^2,
\end{align*}
and the three-term recurrence relation \eqref{ttrr} gives
$$
\langle L_{n}^{(1)} (x), L_{n}^{(1)}(x) \rangle_S = (\lambda+1)   \|L^{(1)}_{n}\|^2
+  \|L^{(2)}_{n-1}\|^2 - \frac{1}{4} (2n+2)  \|L^{(1)}_{n}\|^2.
$$

In this way, from \eqref{norm-Laguerre} we conclude
$$
n+2 = a_n \left( 4 \lambda + 2n +2 \right) - n \, a_n a_{n-1},
$$
and \eqref{rec_rel_an} holds.

Finally, from \eqref{a_n} we get
$$
a_{0} = \frac{1}{4}\frac{\|L^{(1)}_1\|^2}{\langle S_{0}, S_{0} \rangle_S} = \frac{1}{2\lambda + 1 } < 1,
$$
and induction on $n$ gives
$$
a_n < \frac{n+2}{4 \lambda + n  + 2} < 1
$$
for all $n = 0, 1, \ldots$
\end{proof}

\medskip
The coefficients $\{a_n\}_{n\geq 0}$ in \eqref{conechat} can be explicitly expressed in terms of Laguerre polynomias and therefore we can get their first order asymptotic expansion.

\begin{lemma} The sequence $\{a_n\}_{n\geq 0}$ in \eqref{conechat} satisfies
\begin{equation} \label{rat_func_an}
a_n  = \frac{n+2}{n+1} \frac{L^{(1)}_{n}(-4\lambda)}{ L^{(1)}_{n+1}(-4\lambda)}, \quad
\end{equation}
for $n\geqslant 0$, and therefore
\begin{equation} \label{lim_an}
 a_n =1 - \frac{\sqrt{4\lambda}}{\sqrt{n}} +\mathcal{O}(\frac{1}{n}) .
\end{equation}
\end{lemma}
\begin{proof}
First, we consider the sequence of polynomials in $\lambda$, $\{q_n(\lambda)\}_{n\geqslant 0}$, satisfying the recurrence relation
\begin{equation} \label{ttrr_qn}
q_{n+1}(\lambda) = [ 4\lambda +2(n+1)] q_{n}(\lambda) -n(n+1) q_{n-1}(\lambda), \quad n\geqslant 1,
\end{equation}
where  $q_{0}(\lambda) = 1 $ and $q_{1}(\lambda)= 4\lambda +2$.

We see, for $n=0$, that
$$
a_0 = \frac{2}{4 \lambda + 2} = \frac{2 \,q_{0}(\lambda)}{q_{1}(\lambda)}.
$$
Next, we assume
$$
a_{n-1} = \frac{(n+1) q_{n-1}(\lambda)}{q_{n}(\lambda)},
$$
Hence, from \eqref{rec_rel_an}, we have
\begin{align*}
%a_n = &  \frac{n+2}{4 \lambda + 2(n+1) - n a_{n-1}} \\
a_n = &  \frac{n+2}{4 \lambda + 2(n+1) - n  (n+1)\frac{ q_{n-1}(\lambda)}{q_{n}(\lambda)}
} \\
= &  \frac{(n+2)q_{n}(\lambda)}{[4 \lambda + 2(n+1)]q_{n}(\lambda) - n  (n+1) q_{n-1}(\lambda)}.
\end{align*}
and from \eqref{ttrr_qn} we conclude
$$
a_n  = \frac{(n+2) q_{n}(\lambda)}{q_{n+1}(\lambda)}, \quad n\geqslant 0.
$$

The three-term recurrence relation \eqref{ttrr} with $\alpha =1$, becomes
$$ (n+1)\,L_{n+1}^{(1)}(x) = [-x+ 2(n+1) ]L_{n}^{(1)}(x) - (n+1) \,L_{n-1}^{(1)}(x), \quad n\geqslant 1,
$$
with $ L_{0}^{(1)}(x) =1, $ and $ L_{1}^{(1)}(x) = -x+2.$
By comparing  with the three-term recurrence relation \eqref{ttrr_qn}, with $x=-4\lambda$, we obtain
$$
q_n(\lambda) =  n! \, L_{n}^{(1)}(-4\lambda), \quad
n \geqslant 0.
$$
Therefore, for  $n\geqslant 0$,
$$
a_n  = \frac{(n+2) q_{n}(\lambda)}{q_{n+1}(\lambda)}
=  \frac{n+2}{n+1}
\frac{L^{(1)}_{n}(-4\lambda)}{L^{(1)}_{n+1}(-4\lambda)},
$$
and \eqref{rat_func_an} holds.

Finally, \eqref{lim_an} follows from \eqref{rat_func_an} and \eqref{asymp_expan}
with $\alpha = \beta = 1, j = -1,$ and $d=2$.
\end{proof}

\medskip

From \eqref{rat_func_an}, we can deduce a full representation of 
polynomials $S_{n}(x)$ in terms of Laguerre polynomials
Using the expression for $a_n$ given in \eqref{rat_func_an}, relation \eqref{connect} can be written as follows:
$$
L_{n}^{(1)} (x)= S_{n}(x) +
\frac{n+1}{n}
\frac{L^{(1)}_{n-1}(-4\lambda)}{L^{(1)}_{n}(-4\lambda)}
S_{n-1}(x)
$$
and therefore
\begin{equation} \label{Sn-Ln}
\frac{S_{n}(x)L^{(1)}_{n}(-4\lambda)}{n+1}  = \frac{L_{n}^{(1)}(x)L^{(1)}_{n}(-4\lambda)}{n+1} -
\frac{S_{n-1}(x)L^{(1)}_{n-1}(-4\lambda)}{n}.
\end{equation}
Iterating this equality we conclude
\begin{equation} \label{Sn-Ln_full}
\frac{S_{n}(x)L^{(1)}_{n}(-4\lambda)}{n+1} =  \sum_{k=0}^{n} (-1)^{n-k} \frac{L_{k}^{(1)}(x)L^{(1)}_{k}(-4\lambda)}{k+1}.
\end{equation}

\medskip

This relation contains the key for some asymptotic properties.

\begin{theorem}
\begin{equation} \label{asymp_Sn}
\lim_{n\rightarrow+\infty}\frac{S_{n}(x)}{\sqrt{n+1}L^{(1)}_{n}(x)} = 0,
\end{equation}
uniformly in every closed domain of the complex plane with
empty intersection with $(0,+\infty)$.
\end{theorem}

\begin{proof}
Using \eqref{Sn-Ln} twice we get
\begin{align*}
\frac{S_{n}(x)}{L^{(1)}_{n}(x)}  &= 1 -\frac{n+1}{n}
\frac{L^{(1)}(x)L^{(1)}_{n-1}(-4\lambda)}{L_{n}^{(1)}(x)L^{(1)}_{n}(-4\lambda)}+ \frac{n+1}{n-1}\frac{S_{n-2}(x)L^{(1)}_{n-2}(-4\lambda)}{L_{n}^{(1)}(x)L^{(1)}_{n}(-4\lambda)}.
\end{align*}
Let us define 
$$
Y_n(x) := \sqrt{n+1} \frac{S_{n}(x)}{L^{(1)}_{n}(x)},
$$
then, we can write
$$
Y_n(x) = \alpha_n(x) + \beta_n(x) Y_{n-2}(x),
$$
with
\begin{align*}
\alpha_n(x) &= \sqrt{n+1}\left[1 -\frac{n+1}{n}
\frac{L^{(1)}(x)L^{(1)}_{n-1}(-4\lambda)}{L_{n}^{(1)}(x)L^{(1)}_{n}(-4\lambda)}\right],\\ 
\beta_n(x)& = \frac{\sqrt{n+1}}{\sqrt{n-1}}\frac{n+1}{n-1}\frac{L^{(1)}_{n-2}(x)L^{(1)}_{n-2}(-4\lambda)}{L_{n}^{(1)}(x)L^{(1)}_{n}(-4\lambda)}.
\end{align*}
Obviously $\alpha_n(x)$ converges to $\sqrt{-x}+\sqrt{4\lambda}$ uniformly on compact domains of 
the complex plane with empty intersection with $(0,+\infty)$, and therefore there exists a positive constant $C$ (depending on $\lambda$ and the compact domain) such that $|\alpha_n(x)| < C$. On the other hand, $\beta_n(x)$ converges to $1$, but again from \eqref{asymp_expan} we can conclude that there exists a positive integer number $n_0$ such that $|\beta_n| < 1$ for $n\geqslant n_0$. In this way, we can write
$$
|Y_n(x)| \leqslant C + |Y_{n-2}(x)|, \quad n\geqslant n_0,
$$
and we can conclude $|Y_n(x)| \leqslant (n+1) C + K$ for some positive constant $K$, uniformly on every compact subset of $\mathbb{C}\setminus(0,+\infty)$. Thus, the sequence
$$
\frac{1}{n+1}Y_n(x) = \frac{1}{\sqrt{n+1}} \frac{S_{n}(x)}{L^{(1)}_{n}(x)}
$$
is uniformly bounded and \eqref{asymp_Sn} follows taking limits in the identity
$$
\frac{1}{\sqrt{n+1}} = \frac{1}{\sqrt{n+1}} \frac{S_{n}(x)}{L^{(1)}_{n}(x)} +  \frac{1}{\sqrt{n+1}} a_{n-1}\frac{S_{n-1}(x)}{L^{(1)}_{n}(x)}.
$$
\end{proof}

\medskip

The following result provides a generating function for the Laguerre--Sobolev orthogonal polynomials.

\begin{proposition}
\begin{equation} \label{gen_fun}
\sum_{n=0}^{\infty} \frac{S_{n}(x)L^{(1)}_{n}(-4\lambda)}{n+1} \omega^n  =
\frac{1}{1-\omega^2} e^{- \frac{(x-4 \lambda)\omega}{1-\omega}} \frac{1}{2\sqrt{x\lambda\omega}} J_{1} \left(  \frac{4\sqrt{x\lambda\omega)}}{1-\omega}\right).
\end{equation}
\end{proposition}

\begin{proof}
For $|\omega| < 1$, multiply \eqref{Sn-Ln_full} times $\omega^n$ to obtain
$$
\frac{S_{n}(x)L^{(1)}_{n}(-4\lambda)}{n+1} \omega^n =  \sum_{k=0}^{n} (-\omega)^{n-k} \frac{L_{k}^{(1)}(x)L^{(1)}_{k}(-4\lambda)}{k+1} \omega^k.
$$
Hence, the series 
$$
\sum_{n=0}^{\infty} \frac{S_{n}(x)L^{(1)}_{n}(-4\lambda)}{n+1} \omega^n
$$
converges for $|\omega| < 1$ since it is the Cauchy product of two convergent series. 
Moreover, from \eqref{gen_fun_Laguerre} we obtain
$$
\sum_{n=0}^{\infty} \frac{S_{n}(x)L^{(1)}_{n}(-4\lambda)}{n+1} \omega^n  =
\frac{e^{- \frac{(x-4 \lambda)\omega}{1-\omega}}}{(1+\omega)(1-\omega)}  \frac{1}{2\sqrt{x\lambda\omega}} J_{1} \left(  \frac{4\sqrt{x\lambda\omega)}}{1-\omega}\right)
$$
and \eqref{gen_fun} holds.
\end{proof}

%%%%%%%%%%%%%%%%%%%%%%%%%%%%%%%%%%%%%%%%%%%%
\section{
The diagonalized spectral method for BVP
}
%%%%%%%%%%%%%%%%%%%%%%%%%%%%%%%%%%%%%%%%%%%%

Let us consider the Fourier-Sobolev expansion of the solution of the BVP \eqref{BVP} in terms of the orthogonal sequence $\{ S_{n}(x) x e^{-x/2}\}_{n\geqslant0}$
$$
\sum_{n=0}^{+\infty} \widehat{u}_{n}  S_{n}(x) x e^{-x/2},
$$
here
$$
\widehat{u}_n
=\frac{\langle u,  S_{n}(x) x e^{-x/2}\rangle_\lambda}{\| S_{n}(x) x e^{-x/2}\|^2_\lambda}.
$$
%On the one hand, since
%$$\lim_{x\to 0} u(x) = \lim_{x\to +\infty} u(x)= 0,$$
%there exists a function $v(x)$ such that $u(x)= v(x)  x e^{-x/2}$. Then, using \eqref{s1}, we obtain
%$$
%\widehat{u}_{n} =  \frac{\langle  v  x e^{-x/2}, S_n  x e^{-x/2}\rangle_\lambda}{\|S_n  x e^{-x/2}\|^2_\lambda}
%=\frac{\langle v, S_n\rangle_S}{\|S_n\|^2_S} = \widehat{v}_{n},
%$$
%where $\widehat{v}_{n}$ is the $n$th coefficient of the Fourier expansion of $v$ in terms of the Sobolev orthogonal basis $\{S_{n}(x)\}_{n\geqslant0},$ given by
%$$
%\sum_{n=0}^{+\infty} \widehat{v}_{n} S_{n}(x).
%$$

Again, integration by parts gives
\begin{align*}
\widehat{u}_n {\| S_{n}(x) x e^{-x/2}\|^2_\lambda}
=& \lambda \int_{0}^{+\infty} u(x) S_n(x) \frac{1}{x} x e^{-x/2}dx + \int_{0}^{+\infty} u'(x)[S_n(x) x e^{-x/2}]'dx\\
=&\int_{0}^{+\infty} [-u''(x)+\lambda \frac{1}{x} u(x)] S_n(x) x e^{-x/2}dx\\
=&\int_{0}^{+\infty} f(x) S_n(x) x e^{-x/2}dx.
\end{align*}
Therefore to compute  the Fourier-Sobolev coefficients of the solution $u(x)$ of the boundary value problem we do not need to solve any system of linear equations, we need just to compute the integrals
\[  f(n) := \int_{0}^{+\infty} f(x) S_n(x) x e^{-x/2}dx,
\]
and the Sobolev norms \(s(n) = \| S_{n}(x) x e^{-x/2}\|^2_\lambda\).

A  method to recursively generate the sequence $\{{f}(n)\}_{n\geqslant0}$ can be deduced from the conexion formula betweeen Laguerre and Laguerre--Sobolev polynomials, in fact, if we denote
\[
g(n):= \int_{0}^{+\infty} f(x) L^{(1)}_{n}(x) x e^{-x/2}dx
\]
from \eqref{conechat} we get
\[
g(n) = f(n) + {a}_{n-1}  f(n-1),
\]
which provides the announced recursive method assuming the initial condition \[f(0) = g(0) = \int_{0}^{+\infty} f(x) x e^{-x/2}dx.\]

The Sobolev norms \(s(n)\) satisfy a similar recurrence relation. From \eqref{conechat} we get
\[
\| L^{(1)}_{n}(x) x e^{-x/2}\|^2_\lambda = s(n) + {a}_{n-1}^2  s(n-1),
\]
and using \eqref{s1} we conclude  the recurrence
\[
s(n) = (n+1)\left(\lambda+\frac{n+1}{2}\right) - {a}_{n-1}^2  s(n-1),
\]
with initial condition \(s(0) = \lambda + 1/2.\)

\medskip

In general, the partial sums of the Laguerre-Sobolev expansion of the solution of the BVP \eqref{BVP} can be obtained recursively.
Let us denote
\[
\mathcal{S}_n(u,x) = \sum_{k=0}^n \widehat{u}_k S_{k}(x) x e^{-x/2}.
\]
With the initial conditions
\begin{align}
a_0 &= \frac{1}{2 \lambda + 1}, \nonumber\\
f(0) &= g(0) = \int_{0}^{+\infty} f(x) x e^{-x/2}dx,\nonumber\\
s(0) &= \lambda + 1/2,\label{initial_cond}\\
S_0(x)&= 1, \nonumber\\
\mathcal{S}_0(u,x) &= x e^{-x/2}, \nonumber 
\end{align}
the recurrence can be performed as follows
\begin{align}
a_n &= \frac{n+2}{4 \lambda + 2(n+1) - n a_{n-1}}, \nonumber\\
g(n) &= \int_{0}^{+\infty} f(x) L^{(1)}_{n}(x) x e^{-x/2}dx, \nonumber\\
f(n) &= g(n) - {a}_{n-1}  f(n-1), \nonumber\\
s(n) &= (n+1)\left(\lambda+\frac{n+1}{2}\right) - {a}_{n-1}^2  s(n-1),\label{recurrence}\\
S_n(x)&= L^{(1)}_{n}(x) - {a}_{n-1} S_{n-1}(x), \nonumber\\
\mathcal{S}_n(u,x) &= \mathcal{S}_{n-1}(u,x) + \frac{f(n)}{s(n)} S_n(x) x e^{-x/2}. \nonumber
\end{align}
Notice that in this computation the only integrals are the values of $g(n)$, which simply depends on $f(x)$ and $L^{(1)}_{n}(x)$. Exactly the same integrals we need to compute in a spectral method based in Laguerre polynomials.

%%%%%%%%%%%%%%%%%%%%%%%%%%%%%%%%%%%%%%%%%%%%
\section{
Numerical experiment
}
%%%%%%%%%%%%%%%%%%%%%%%%%%%%%%%%%%%%%%%%%%%%

In this section, we examine some numerical experiments aimed at assessing the reliability and accuracy of the Sobolev spectral method for addressing Dirichlet boundary problems over the interval \((0,+\infty)\).

We focus on a second-order Dirichlet boundary value problem linked to the non-homogeneous Schrödinger equation with a potential given by \(V(x) = 1/x\)
\begin{align*}
&- u''(x) + \lambda \frac{1}{x} u(x) = f(x), \\
& u(0) = u(+\infty) = 0,
\end{align*}
and $\lambda = 1$.

According to Boyd \cite{Boyd01}, the convergence of an approximation developed through Laguerre polynomials is influenced by the behavior of the solution as it approaches infinity. Therefore, we can not predict a better performance for the Laguerre--Sobolev approximants in that sense. But as we can expect, spectral convergence of an approximation developed through Laguerre--Sobolev polynomials is achieved for solutions that exhibit exponential decay towards zero. To illustrate this, we analyze the case where $f(x)= e^{-x} (3 \cos{x} - 2 (-1 + x) \sin{x})$. Notice that the solution to the boundary value problem is $u(x) = x \cos{x} e^{-x}$.

The sequence of orthogonal functions $\{ S_{n}\, x\,e^{-x/2}\}_{n\geqslant0}$ orthogonal with respect to the Sobolev inner product
$$
\langle u, v\rangle_\lambda = \int_{0}^{+\infty}\lambda \frac{1}{x} u(x)\,v(x)\,dx + \int_{0}^{+\infty} u'(x)\,v'(x)\,dx
$$
can be easily computed applying the recurrence \eqref{initial_cond} and \eqref{recurrence}. In particular, the first five Sobolev polynomials are

\begin{align*}
S_0(x) &= 1,\\
S_1(x) &= \frac{5}{3}-x,\\
S_2(x) &= \frac{x^2}{2}-\frac{60 x}{23}+\frac{54}{23},\\
S_3(x) &= -\frac{x^3}{6}+\frac{189 x^2}{106}-\frac{258
   x}{53}+\frac{158}{53},\\
S_4(x) &= \frac{x^4}{24}-\frac{1285 x^3}{1701}+\frac{25 x^2}{6}-\frac{1460 x}{189}+\frac{2045}{567}.
\end{align*}

Let us denote by
$$
\sum_{n=0}^{+\infty} \hat{u}_{n} S_{n}\, x\,e^{-x/2}
$$
the Fourier--Sobolev expansion of $u(x)$, where the coefficients $\hat{u}_{n}$ satisfy
\begin{align*}
\hat{u}_n \|S_{n}\, x\,e^{-x/2}\|_\lambda^2
=& \int_{0}^{+\infty}\lambda \frac{1}{x} u(x) S_{n}\, x\,e^{-x/2} dx + \int_{0}^{+\infty} u'(x)[\int_{0}^{+\infty}]'dx\\
=&\int_{0}^{+\infty} [-u''(x)+ \lambda \frac{1}{x} u(x)]S_{n}\, x\,e^{-x/2}dx\\
=&\int_{0}^{+\infty} f(x) S_{n}\, x\,e^{-x/2} dx.
\end{align*}

The partial sums of the Laguerre--Sobolev expansion of the solution of the BVP can be obtained using the recurrence 
described in \eqref{recurrence}, with initial conditions \eqref{initial_cond}.

\begin{figure}[ht]
\centerline{\includegraphics[scale=0.75]{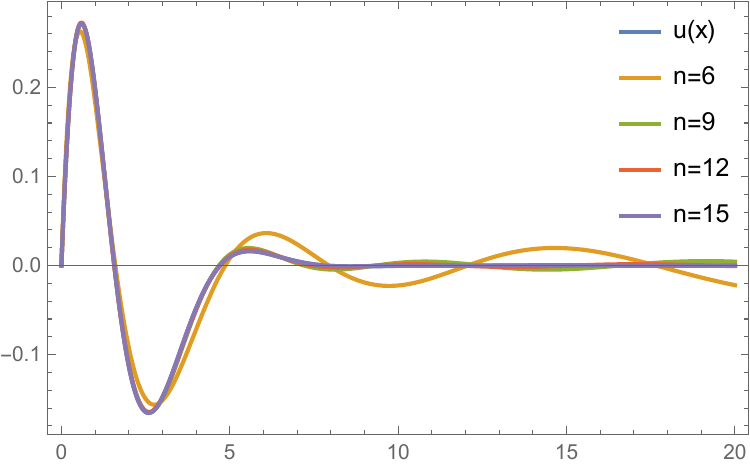}}
\caption{\label{fig1}The solution $u(x)$ and the Fourier-Sobolev partial sums $\mathcal{S}_n(u,x)$ for $n = 6, 9, 12, 15$.}
\end{figure}

In Figure~\ref{fig1}, we plot the solution $u(x)$ and the  corresponding Fourier--Sobolev partial sums
\[
\mathcal{S}_n(u,x) = \sum_{k=0}^n \widehat{u}_k S_{k}(x) x e^{-x/2},
\]
for $n = 6, 9, 12, 15$.

\begin{figure}[ht]
\centerline{\includegraphics[scale=0.75]{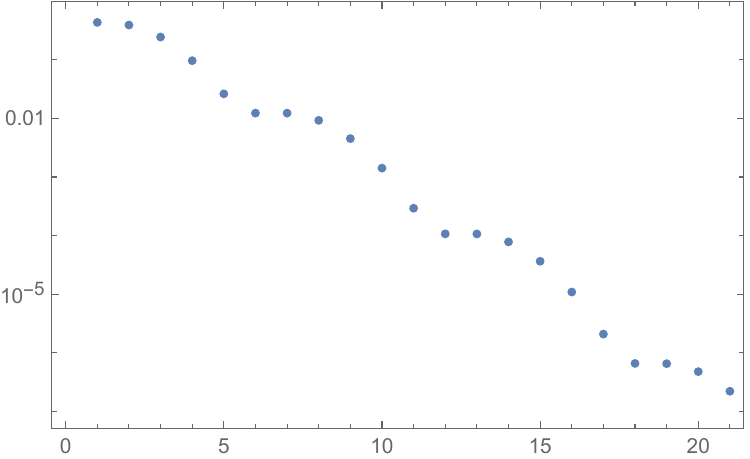}}
\caption{\label{fig2}A logarithmic plot of errors for $n = 0, 1, \ldots, 20$.}
\end{figure}

In Figure~\ref{fig2}, we show a logarithmic plot of the square of the errors in the Sobolev norm
$$ \epsilon_n = \|u(x)-\mathcal{S}_n(u,x)\|_\lambda^2,$$
for $n = 0, 1, \ldots, 20$. The plot clearly indicates an exponential convergence rate.
Observe that this happens not only for the Fourier-Sobolev partial sums but also for their first derivatives.

\medskip

In our following example we consider a solution that does not decay exponentially to infinity
\[
u(x) = 10 \frac{x \cos{x}}{(x + 1)^3}.
\]
In this case
\[
f(x) = 10 \frac{(7 + x (-3 + x (3 + x))) \cos{x} - 2 (-1 + x + 2 x^2) \sin{x}}{(x + 1)^5}.
\]
Again, we have computed the Laguerre--Sobolev approximants 
for $n = 0, 1, \ldots, 20$. 

\begin{figure}[ht]
\centerline{\includegraphics[scale=0.75]{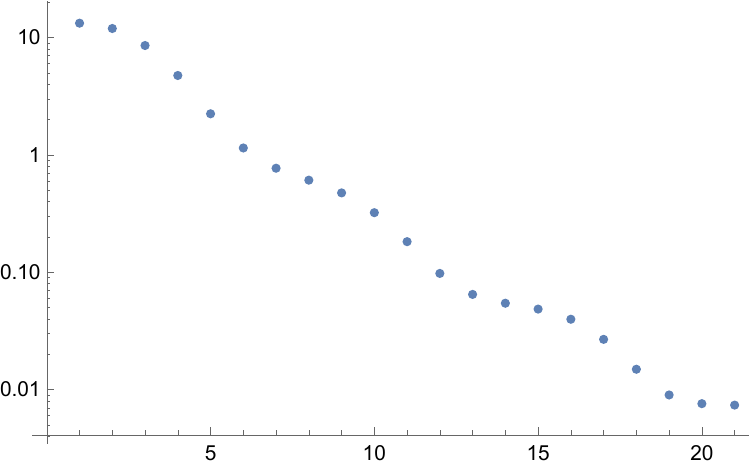}}
\caption{\label{fig3}A logarithmic plot of errors for $n = 0, 1, \ldots, 20$ for a solution with rational decay.}
\end{figure}

In Figure~\ref{fig3}, we show a logarithmic plot of the square of the errors in the Sobolev norm $$ \epsilon_n = \|u(x)-\mathcal{S}_n(u,x)\|_\lambda^2,$$
for $n = 0, 1, \ldots, 20$. The graph clearly indicates worse performance from the approximants.

%%%%%%%%%%%%%%%%%%%%%%%%%%%%%%%%%%%%%%%%%%%%%%%%%%%%%%%%%%%%%%%%%%%%%

\end{document}